\documentclass{article}
\usepackage[latin1]{inputenc}
\usepackage[T1]{fontenc}
\usepackage{amsmath,amssymb}
\usepackage{amsfonts,amsthm}
\usepackage{geometry}
\usepackage{graphicx}
\usepackage[plainpages=false,colorlinks=true, breaklinks=true]{hyperref}

\hypersetup{linkcolor=blue, citecolor=red, menucolor=blue, urlcolor=blue}

\newcommand{\Cb}{\mathbb{C}}

\newcommand{\Ib}{\mathbb{I}}

\newcommand{\Nb}{\mathbb{N}}

\newcommand{\Rb}{\mathbb{R}}

\newcommand{\Tb}{\mathbb{T}}
\newcommand{\Zb}{\mathbb{Z}}

\newcommand{\Xb}{\textbf{\upshape X}}

\newcommand{\Ac}{\mathcal{A}}
\newcommand{\Bc}{\mathcal{B}}

\newcommand{\Fc}{\mathcal{F}}

\newcommand{\Kc}{\mathcal{K}}
\newcommand{\Lc}{\mathcal{L}}
\newcommand{\Mc}{\mathcal{M}}
\newcommand{\Pc}{\mathcal{P}}
\newcommand{\Qc}{\mathcal{Q}}

\newcommand{\Uc}{\mathcal{U}}

\newcommand{\Wc}{\mathcal{W}}

\newcommand{\Alp}{\mathcal{A}_{p}}
\newcommand{\Rh}{\hat{R}}
\newcommand{\ii}{{\bf i}}

\DeclareMathOperator{\coker}{coker}
\DeclareMathOperator{\diag}{diag}
\DeclareMathOperator{\im}{im}
\DeclareMathOperator{\ind}{ind}
\DeclareMathOperator{\op}{op}
\DeclareMathOperator{\plim}{\Pc-lim}

\DeclareMathOperator{\rk}{rank}

\DeclareMathOperator{\plimn}{\underset{n\to\infty}{\Pc-lim \,}}

\providecommand{\lb}[1]{\Lc(#1)}
\providecommand{\lc}[1]{\Kc(#1)}
\providecommand{\pb}[1]{\Lc(#1,\Pc)}
\providecommand{\pc}[1]{\Kc(#1,\Pc)}

\newcount\comment

\newtheorem{thm}{Theorem}

\newtheorem{lem}[thm]{Lemma}
\newtheorem{prop}[thm]{Proposition}
\newtheorem{cor}[thm]{Corollary}

\theoremstyle{definition}
\newtheorem{defn}[thm]{Definition}

\newtheorem{ex}[thm]{Example}

\begin{document}

\title{Fredholm theory for band-dominated\\ and related operators: a survey}
\author{Markus Seidel}
\maketitle

\begin{abstract}
This paper presents the Fredholm theory on $l^p$-spaces for band-dominated operators 
and important subclasses, such as operators in the Wiener algebra. It particularly 
closes several gaps in the previously known results for the case $p=\infty$ and
addresses the open questions raised by Chandler-Wilde and Lindner \cite{LiChW}.
The main tools are provided by the limit operator method and an algebraic framework 
for the description and adaption of Fredholmness and convergence.
A comprehensive overview of this approach is given.


\end{abstract}

\

\section{Introduction}

During the last years a far reaching theory on band-dominated operators grew up
and revealed deep results concerning their Fredholm property, their spectral properties, 
the applicability (and further side effects) of the finite section method, and was 
successfully applied to several more concrete subclasses of operators having 
more advanced properties. Among them one can find Toeplitz, Hankel or Jacobi operators, 
hence also discrete Schroedinger operators.

Besides, the wonderful concept of operator algebras arising from approximate projections 
and the limit operator method were developed. Today one may say that large parts of 
this theory are very well understood and finalized. However, there are some results 
which still require some additional assumptions that seem to be redundant but could 
not be removed completely yet. One prominent example is the need for a predual setting 
in case of band-dominated operators on $l^\infty$-spaces in the works of Lindner. A 
collection of eight open problems was stated by Chandler-Wilde and Lindner in the 
final chapter of \cite{LiChW}.

The aim of the present text is to give an overview on the latest state of the art 
and to close several gaps in the theory on band-dominated operators. We show that 
a couple of results being known for band-dominated operators actually hold in a more 
general context, and we try to clear up some possible generalizations which have 
already been indicated in the literature. We particularly contribute to seven 
of the eight open questions stated in \cite{LiChW} answering four of them completely. 

\subsection{Band-dominated operators}
Let us start introducing the basic notions and the precise description of two of 
the mentioned (and actually redundant) conditions \eqref{CPrae} and \eqref{CHyper}.

\subparagraph{Sequence spaces}
Given $N\in\Nb$, a Banach space $X$, and the parameter $1\leq p < \infty$ we 
let $l^p=l^p(\Zb^N,X)$ denote the space of all functions $x:\Zb^N\to X$ with the 
property 
\[\|x\|_p:=\left(\sum_{i\in\Zb^N}\|x(i)\|^p\right)^{1/p}<\infty.\] 
Provided with the norm $\|x\|_p$, $l^p$ becomes a Banach space.
It is convenient to refer to such functions as sequences $(x_i)_{i\in\Zb^N}$ with 
$x_i=x(i)$. The above condition then simply means that the sequences shall be 
$p$-summable.
Analogously, one introduces the respective Banach spaces $l^\infty=l^\infty(\Zb^N,X)$ 
of all bounded sequences $x=(x_i)$ of elements $x_i\in X$ with the norm
$\|\cdot\|_\infty$ defined by 
\[\|x\|_\infty:=\sup\{\|x_i\|:i\in\Zb^N\}\] 
and, finally, its closed subspace $l^0=l^0(\Zb^N,X)$ consisting of all bounded 
sequences $(x_i)$ with $\|x_i\|\to 0$ as $|i|\to\infty$.

Note that, choosing $X=L^p((0,1)^N)$, $l^p(\Zb^N,X)$ is isometrically isomorphic to
$L^p(\Rb^N)$. Hence, all subsequent definitions and results for the ``discrete'' 
$l^p$-cases have their ``continuous'' $L^p$-counter\-parts. This is a well known and
frequently utilized observation (see e.g. 
\cite{RaConv, RaRoSiL2, LiEssBdd, LpInd, LimOps, Marko, LiChW}), and we will not 
go into further details here.

\subparagraph{The operators}
Every sequence $a=(a_i)\in l^\infty(\Zb^N,\lb{X})$, where $\lb{X}$ denotes the 
Banach algebra of all bounded linear operators on $X$, gives rise to a bounded 
linear operator $aI$ of multiplication on each of the spaces $l^p$ by 
$(x_i)\mapsto(a_ix_i)$. 

Another basic family of operators in $\lb{X}$ is given by the shifts
\[V_\alpha :l^p\to l^p,\quad (x_i)\mapsto(x_{i-\alpha}) \quad\quad(\alpha\in\Zb^N).\]

This is everything we need for the definition of band-dominated operators:
\begin{defn}
A finite sum of the form $\sum_\alpha a_\alpha V_\alpha$ is called a band operator. 
The limits of sequences of band operators, taken w.r.t. the operator norm, are said 
to be band-dominated, and the set of all band-dominated operators on $l^p$ shall 
be denoted by $\Alp$.
\end{defn}
Clearly, band operators form a (non-closed) algebra $\Bc$ of bounded linear operators 
on each of the spaces $l^p$, $p\in\{0\}\cup[1,\infty]$. Thus, $\Alp$ is a Banach algebra
for every $p$. Note that $\Bc$ is independent of the choice of the parameter $p$, whereas 
$\Alp$ depends on $p$. 

For a nice introduction and a comprehensive discussion we refer to the work of 
Lindner (e.g. \cite{Marko}, \cite{LiChW}) and the book of Rabinovich, Roch and 
Silbermann \cite{LimOps}.

\subsection{The Wiener algebra}
A prominent algebra which is somehow between $\Bc$ and $\Alp$ and which provides 
a remarkable Fredholm behavior is the so-called Wiener algebra $\Wc$:
For band operators $A=\sum_\alpha a_\alpha V_\alpha$ set 
\[\|A\|_{\Wc}:=\sum_\alpha\|a_\alpha\|_\infty\]
and notice that $\|\cdot\|_{\Alp}\leq \|\cdot\|_{\Wc}$.
Now take $\Wc$ as the closure of $\Bc$ with respect to $\|\cdot\|_\Wc$. Equipped 
with $\|\cdot\|_\Wc$ as the norm, $\Wc$ actually turns into a Banach algebra which 
is a subset of each $\Alp$, $p\in\{0\}\cup[1,\infty]$. To see that it is even a 
proper subset, check that e.g. $B(x_i):=((|i|+1)^{-1}x_{-i})$ defines an operator 
$B\in\Alp\setminus\Wc$.

Note that $\Wc$ can be regarded as a natural (non-stationary) extension of the classical 
algebra of all (Laurent) operators with constant diagonals and the norm 
$\|\cdot\|_{\Wc}$, which is isomorphic to the Wiener algebra of all periodic functions 
with absolutely summable sequence of Fourier coefficients. For further details we 
refer to \cite[Section 1.3.6]{Marko}. 

In \cite{MarkoWiener} Lindner proved that the Fredholm properties of the operators 
in the Wiener algebra are independent of the underlying space. By this he extended a
series of predecessors in \cite{LangeR, RochFredTh, LimOps, Rochlp, LpInd}. 
Actually, for the proofs in \cite{MarkoWiener} two additional restrictions were still 
required:
\begin{equation}\label{CPrae}
\begin{minipage}{10cm}
Say that for $A\in\lb{l^\infty(\Zb^N,X)}$ there exists a \textbf{predual setting} 
if there is a Banach space $Y$ and an operator $B\in\lb{l^1(\Zb^N,Y)}$ such that 
such that $Y^*=X$ and $B^*=A$.
\end{minipage}
\end{equation}
\begin{equation}\label{CHyper}
\begin{minipage}{10cm}
Say that the Banach space $X$ has the \textbf{hyperplane property} if it is
isomorphic to one of its subspaces of co-dimension one.
Equivalently, $X$ has the hyperplane property if there is a $B\in\lb{X}$ 
of Fredholm index one.
\end{minipage}
\end{equation}
Now, we can state Lindners result as it appears in \cite{MarkoWiener} and 
\cite[Theorem 6.44]{LiChW}:
\begin{thm}\label{TWiener}
Suppose that $X$ is finite-dimensional or has the hyperplane property, and that 
$A\in\Wc$. Then
\begin{enumerate}
\item[(a)] If $A$ is Fredholm on one of the spaces $l^p$ with $p\in\{0\}\cup[1,\infty)$, 
					 then $A$ is Fredholm on all spaces $l^p$, $p\in\{0\}\cup[1,\infty]$.
\item[(b)] If for $A$ considered as acting on $l^\infty$ there exists a predual setting,
					 then $A$ is Fredholm on one of the spaces $l^p$ if and only if $A$ is 
					 Fredholm on all spaces $l^p$, $p\in\{0\}\cup[1,\infty]$.
\end{enumerate}
If $A$ is Fredholm on every space $l^p$ then the index is the same on all these spaces.
\end{thm}
The Open Problems No. 4 and 5 in \cite{LiChW} ask whether the existence of a 
predual setting or the hyperplane property are redundant, and we will answer both 
questions affirmatively within this text. This particularly simplifies Theorem 
\ref{TWiener} as follows:
\begin{thm}\label{TWiener2}
An operator $A\in\Wc$ is Fredholm on one of the spaces $l^p$ if and only if it is
Fredholm on all the spaces $l^p$, $p\in\{0\}\cup[1,\infty]$. In this case the index 
is the same on all the spaces.
\end{thm}

Actually, the need for a predual setting played an important role not only for the 
treatment of operators in the Wiener algebra but 
it affected the whole general theory on Fredholmness and limit operators 
which we are going to discuss within this paper.

\subsection{Approximate projections}
One of the most fruitful investigations for the development of the theory of 
band-dominated operators is the concept of approximate projections $\Pc=(P_n)$ and the 
substitution of the classical triple (compactness, Fredholmness, strong convergence) 
by ($\Pc$-compactness, $\Pc$-Fredholmness, $\Pc$-strong convergence).
This idea has its roots in the work of Simonenko on operators of local type \cite{Sim},
grew up in the work of Roch and Silbermann \cite{NonStrongly}, \cite{LimOps} and
became an indispensable tool in the Fredholm theory of band-dominated operators
(see e.g. \cite{LimOps}, \cite{Marko} and numerous papers which led these monographs or 
followed them). Unfortunately, also in this concept there had been open 
questions (e.g. on the connections between $\Pc$-Fredholmness, invertibility at 
infinity \footnote{The definitions follow below.} and the usual Fredholm property) 
which caused some gaps and required additional conditions in some results and their 
applications like the above mentioned condition \eqref{CPrae} on the existence of a 
predual setting. Recently, these problems could be solved \cite{SeSi3}, and in what 
follows we give an overview on the latest state of the art.

\medskip
Whenever the situation is unambiguous we abbreviate the spaces of interest
$l^p(\Zb^N,X)$, with $N\in\Nb$ and $X$ being a Banach space, simply by $\Xb$.

\medskip
As a foretaste of what is to come we announce an operator algebra $\pb{\Xb}$ which
marries up with the new triple ($\Pc$-compactness, $\Pc$-Fredholmness, $\Pc$-strong 
convergence) as $\lb{\Xb}$ does with the classical one, and by this it provides
a selfcontained ``universe'' which amazingly reflects and improves what we know from
$\lb{\Xb}$: This set $\pb{\Xb}$ is an inverse closed Banach subalgebra of $\lb{\Xb}$. 
It is closed with respect to $\Pc$-strong convergence, that is the $\Pc$-strong limit 
of a sequence in $\pb{\Xb}$ always belongs to $\pb{\Xb}$ again, and there are Banach 
Steinhaus type results for the $\Pc$-strong convergence. Multiplication with 
$\Pc$-compact operators turns $\Pc$-strongly converging sequences into norm converging
sequences. The $\Pc$-compact operators form a closed two-sided ideal in $\pb{\Xb}$, 
hence permit to introduce the $\Pc$-Fredholm property as ``invertibility up to 
$\Pc$-compact operators''. $\pb{\Xb}$ proves to be closed under passing to 
$\Pc$-regularizers, and the usual Fredholm property perfectly aligns in that new 
framework. 

Moreover, it should be mentioned that one can carry over the theory on the stability 
of strongly converging approximation methods to $\Pc$-strongly converging methods, 
and actually this even turns out to be the more natural framework for such questions 
in a sense. This is not subject of the present paper, but can be found in e.g.
\cite{NonStrongly, RaRoSiL2, LimOps, Marko, SeSi2, SeSi3, MaSaSe}.

\medskip
This work is organized as follows: The second part is devoted to the mentioned algebraic 
framework which provides the tools for the study of convergence and the Fredholm
properties of the elements in $\pb{\Xb}$. In Section \ref{SBDO} we discuss the application
of these techniques to band-dominated operators and certain subclasses. In 
particular Theorem \ref{TWiener2} is proved there. 
Some possible extensions, generalizations and questions concerning the world beyond 
$\Alp$ and $\pb{\Xb}$ are discussed in the final Section \ref{SGen}.

\section{The \texorpdfstring{$\Pc$}{P}-theory and the limit operator method}
\label{SAP}

\subsection{The approximate projection \texorpdfstring{$\Pc$}{P}}
Given a set $U\subset\Zb^N$ we define $P_U$ acting as operator on $\Xb=l^p(\Zb^N,X)$ 
by
\[x=(x_i)\mapsto (P_U x)_i:=\begin{cases} x_i&\text{ if } i\in U\\0&
\text{ if } i\notin U.\end{cases}\]
Clearly, $P_U$ and $Q_U:=I-P_U$ are complementary projections. The most important 
operators among them are the canonical projections $P_k:=P_{\{-k,\ldots,k\}^N}$
and $Q_k:=I-P_k$ with $k\geq 0$.

\medskip
In the following we build upon the sequence
\begin{equation}\label{EAppProj}
\Pc:=(P_1,P_2,P_3,\ldots),
\end{equation}
which has the following properties
\begin{itemize}
\item $P_n\neq 0$, $P_n\neq I$ for every $n$, the $P_n$ are uniformly bounded
			and for every $m$ there is an $N_m$  such that $P_nP_m=P_mP_n=P_m$ whenever 
			$n\geq N_m$ (actually $N_m:=m$ already does the job in the present case),
\item $\sup\|P_U\|<\infty$, the supremum over all finite subsets $U$ of $\Zb^N$,
\item $\sup_n\|P_nx\|\geq\|x\|$ for every $x\in\Xb$.
\end{itemize}
Sequences $\Pc=(P_n)$ of operators with the first property are referred to as 
approximate projections, in e.g. \cite{LimOps, SeSi3}. If the second property is 
fulfilled, then one says that $\Pc$ is uniform, and the third one makes $\Pc$ to 
an approximate identity. Roughly speaking, $\Pc$ then forms a nested sequence of 
operators which permit to explore the whole space $\Xb$ in an asymptotic sense. 
Although the subsequent theory was developed for uniform approximate identities 
in general, we restrict ourselves for simplicity to the particular canonical 
choice \eqref{EAppProj} within this paper.

In the classical theory, e.g. for $p\in(1,\infty)$ and $X=\Cb^K$, one heavily 
exploits the observations that $\Pc$ consists of compact projections and converges 
strongly to the identity, as well as the fact that the multiplication by compact 
operators always turns a strongly convergent sequence $(A_n)$ into a norm convergent 
one $(A_nK)$. Unfortunately, these classical arguments break down if $\dim X=\infty$ 
or $p=\infty$.
Therefore, the $\Pc$-theory turns the table, and take the sequence $\Pc$ (for arbitrary 
$p$ and $X$) as a starting point for the definition of adapted notions of 
$\Pc$-compactness and $\Pc$-strong convergence, which mimic the behavior that is 
known from the classical world, and translate it into a more general and more flexible 
framework.
 
\subsection{\texorpdfstring{$\Pc$}{P}-compact operators}
A bounded linear operator $K$ on $\Xb$ is called $\Pc$-compact if 
\[\|KP_n - K\|\to 0 \quad \text{and} \quad \|P_nK - K\|\to 0 
	\quad\text{as\quad} n \to\infty.\]
By $\pc{\Xb}$ we denote the set of all $\Pc$-compact operators on $\Xb$ and by 
$\pb{\Xb}$ the set of all bounded linear operators $A$ for which $AK$ and $KA$ are 
$\Pc$-compact whenever $K$ is $\Pc$-compact. One may say that $\pb{\Xb}$ collects
the operators which are compatible with $\pc{\Xb}$.

\begin{prop}(\cite[Proposition 1.1.8]{LimOps}) \label{PLP}\\
The set $\pb{\Xb}$ is a closed subalgebra of $\lb{\Xb}$, it contains the identity 
operator $I$, and $\pc{\Xb}$ is a proper closed ideal of $\pb{\Xb}$. 
An operator $A\in\lb{\Xb}$ belongs to $\pb{\Xb}$ if and only if, for every $k\in\Nb$,
\begin{equation}\label{ELXP}
\|P_k A Q_n\| \to 0 \quad \text{and} \quad \|Q_n A P_k\| \to 0 
	\quad\text{as}\quad n \to \infty.
\end{equation}
\end{prop}

At this point we shall rest for a moment just to shortly visualize the relations
between the operator algebras which have been introduced so far. For this we borrow
Figure \ref{Fig1} from \cite[Figure 1]{Marko}. Here $\lc{\Xb}$ denotes the
set of compact operators as usual. 
\begin{figure}[ht]
\centering
	\includegraphics[scale=0.8]{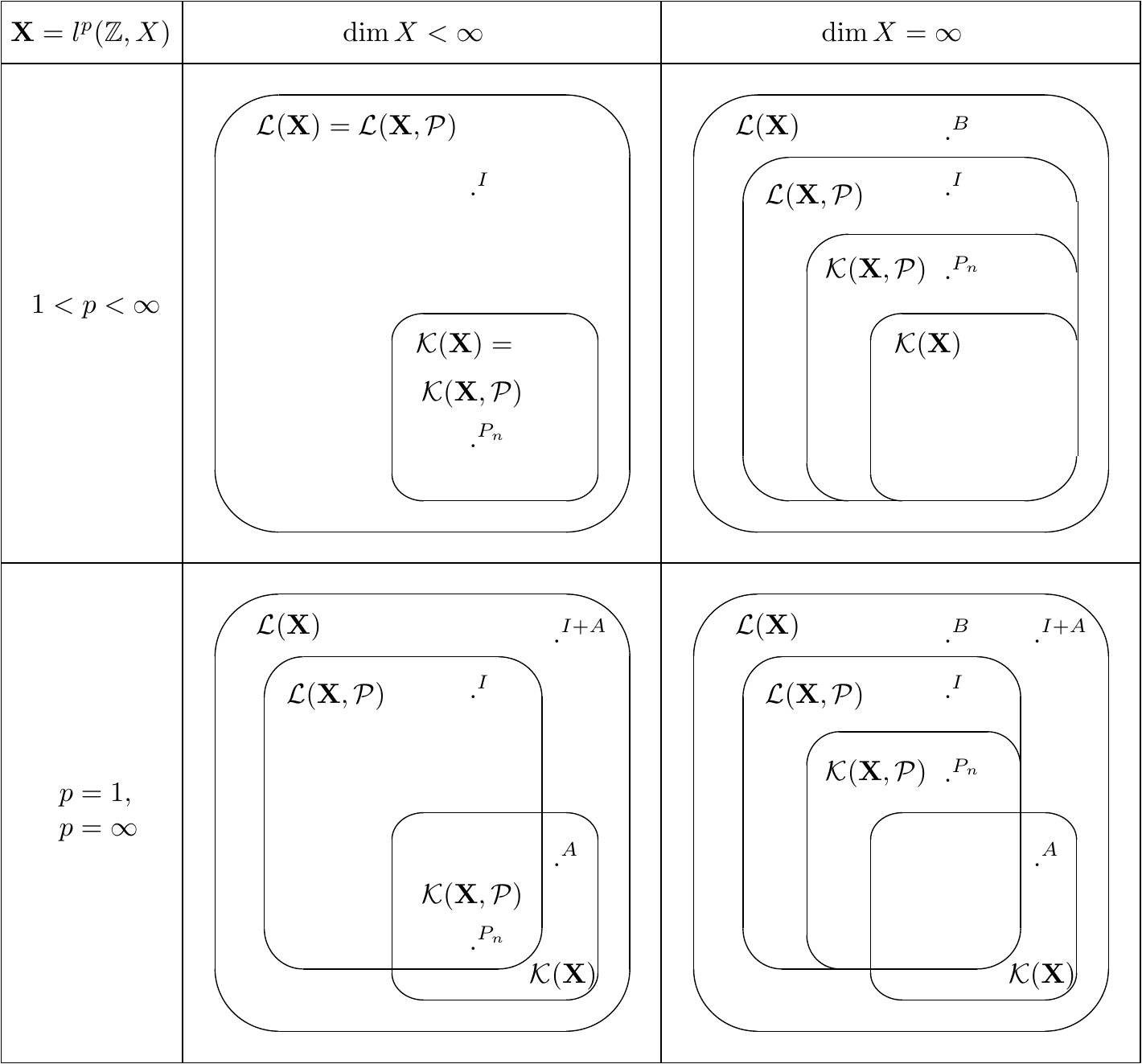}
\caption{Venn diagrams depending 
					on $\Xb=l^p(\Zb,X)$.}
\label{Fig1}
\end{figure}

\begin{ex}
Also the following examples of operators which substantiate the proper inclusions in 
this picture are taken from \cite{Marko}.
\begin{itemize}
\item Let $a\in X$ and $f\in X^*$ be non-zero elements and $A\in\lb{l^1}$ be 
			given by the rule
			\[A: (x_i) \mapsto \left(\ldots,0,0,\sum_i f(x_i)a, 0,0,\ldots\right).\]
			This operator is compact, but does not belong to $\pb{l^1}$ due to the 
			characterization \eqref{ELXP}. 
			The adjoint of $A$ provides the same outcome for the case $p=\infty$.
\item Let $X$ be the space $L^p(0,1)$ of all $p$-Lebesgue integrable functions over
			the interval $(0,1)$ and define $\tilde{B}:l^p(\Zb,X)\to X,\, (u_i)\mapsto v$
			by 
			\[v(x):=\begin{cases}
				u_k(x) & : \text{ if } \displaystyle x\in \left(1-\frac{1}{2^{k-1}},1-\frac{1}{2^{k}}\right)
										\text{ for one }k\in\Nb\\
				0 & : \text{ otherwise.}
			\end{cases}\]
			(As customary we let $L^\infty(0,1)$ stand for the space of all Lebesgue 
			measurable functions that are essentially bounded.)
			Then the linear operator 
			$B:(u_i)\mapsto \left(\ldots,0,0,\tilde{B}(u_i), 0,0,\ldots\right)$
			acts boundedly on $l^p=l^p(\Zb,X)$ for every $p$, but does not belong to 
			$\pb{l^p}$ in any case.
\end{itemize}
\end{ex}

\subsection{\texorpdfstring{$\Pc$}{P}-strong convergence}
A sequence $(A_n)\subset\lb{\Xb}$ is said to converge $\Pc$-strongly to $A\in\lb{\Xb}$ 
if, for every $K\in\pc{\Xb}$, both $\|(A_n-A)K\| \text{ and } \|K(A_n-A)\|$ tend to $0$ 
as $n\to\infty$. In this case we write $A_n \to A$ $\Pc$-strongly or 
$A = \plim_n A_n$.
By $\Fc(\Xb,\Pc)$ we denote the set of all bounded sequences 
$(A_n)\subset \lb{\Xb}$, which possess a $\Pc$-strong limit in $\pb{\Xb}$.

Notice that a simple calculation reveals that (cf. \cite[Proposition 1.1.14]{LimOps})
a bounded sequence $(A_n)\subset \lb{\Xb}$ converges $\Pc$-strongly  to $A\in\lb{\Xb}$
if and only if both $\|(A_n-A)P_m\| \text{ and } \|P_m(A_n-A)\|$ 
tend to $0$ as $n\to\infty$ for every $m$.
\begin{prop}\label{PPstrong}
(\cite[Corollary 1.1.16 et seq.]{LimOps} or \cite[Theorem 1.13]{SeSi3})
\begin{itemize}
\item The algebra $\pb{\Xb}$ is closed with respect to $\Pc$-strong convergence, this 
			means that if a sequence $(A_n)\subset\pb{\Xb}$ converges $\Pc$-strongly to $A$ then 
			$A\in\pb{\Xb}$. Moreover, $(A_n)$ is bounded in this case, and hence belongs 
			to $\Fc(\Xb,\Pc)$.
\item The $\Pc$-strong limit of every $(A_n)\in\Fc(\Xb,\Pc)$ is uniquely determined.
\item Provided with the linear operations
			$\alpha (A_n) + \beta (B_n) := (\alpha A_n + \beta B_n)$, the multiplication 
			$(A_n) (B_n):=(A_n B_n),$ and the norm $\|(A_n)\|:=\sup_n\|A_n\|$,
			$\Fc(\Xb,\Pc)$ becomes a Banach algebra with identity $\Ib:=(I)$. 
			The mapping $\Fc(\Xb,\Pc) \to \pb{\Xb}$ which sends $(A_n)$ to its limit 
			$A=\plim_n A_n$ is a unital algebra homomorphism and
			\begin{equation}\label{EPLim}
			\|A\| \leq \liminf_{n\to\infty} \|A_n\|.
			\end{equation}
\end{itemize}
\end{prop}

\subsection{\texorpdfstring{$\Pc$}{P}-Fredholm operators}

As a third part for our new triple we are going to define an appropriate substitute 
for the Fredholm property of operators. But, as a start, let us first recall the most 
important facts about the classical notion from any textbook on functional analysis, 
or e.g. \cite{GoKru, SeSi3}.

\subparagraph{Fredholm operators}
An operator $A\in\lb{\Xb}$ is said to be Fredholm, if its kernel and cokernel
\[\ker A:=\{x\in\Xb: Ax=0\},\quad \coker A :=\Xb/\im A\]
are of finite dimension, where $\im A:=\{Ax:x\in\Xb\}$ denotes the range of $A$.
That is the case if and only if there is a $B\in\lb{\Xb}$
such that $AB-I$ and $BA-I$ are compact operators. Therefore it is equivalent to
$A+\lc{\Xb}$ being invertible in the Calkin algebra $\lb{\Xb}/\lc{\Xb}$.

If $A$ is Fredholm then $\ind A:=\dim\ker A - \dim\coker A$ is called the index of $A$.
Moreover, $A+B$ is Fredholm for $B$ being compact or $B$ having sufficiently
small norm. If both $A$ and $B$ are Fredholm operators then the product $AB$ is 
Fredholm as well and $\ind AB =\ind A+\ind B$. The latter is called Atkinsons theorem.

\subparagraph{\texorpdfstring{$\Pc$}{P}-Fredholm operators and invertibility at infinity}
Now, replacing compact operators by $\Pc$-compact operators, we get two
possible definitions.
\begin{defn}\label{DInvAtInf}
An operator $A\in\lb{\Xb}$ is said to be invertible at infinity if there is a
$B\in\lb{\Xb}$ such that $AB-I$ and $BA-I$ are $\Pc$-compact operators.
In this case $B$ is referred to as a $\Pc$-regularizer for $A$.
\end{defn}

\begin{defn}\label{DPFred}
An operator $A\in\pb{\Xb}$ is called $\Pc$-Fredholm if the  
coset $A+\pc{\Xb}$ is invertible in the quotient algebra $\pb{\Xb}/\pc{\Xb}$.
\end{defn}
Notice that invertibility at infinity is defined for all bounded linear operators,
whereas for $\Pc$-Fredholmness we are restricted to $\pb{\Xb}$, since we need that 
$\pc{\Xb}$ forms a closed ideal in $\pb{\Xb}$. Both notions have been known and
have been studied for many years. The monographs by Rabinovich, Roch and Silbermann 
\cite{LimOps} and Lindner \cite{Marko} already contain a comprehensive theory and 
many applications of this approach. Nevertheless, it has been an open problem for a 
long time whether these two notions coincide for the operators $A\in\pb{\Xb}$.
At least the inverse closedness of $\pb{\Xb}$ in $\lb{\Xb}$ has been known 
\cite[Theorem 1.1.9]{LimOps}, based on a proof of Simonenko \cite{Sim}.
An affirmative answer to the general question was given recently by the 
author.

\begin{thm}\label{TPFredh}(\cite[Theorem 1.16]{SeSi3})\\
An operator $A\in\pb{\Xb}$ is $\Pc$-Fredholm if and only if it is invertible at 
infinity. In this case every $\Pc$-regularizer of $A$ belongs to $\pb{\Xb}$.
Particularly, $\pb{\Xb}$ is inverse closed in $\lb{\Xb}$.
\end{thm}

Its proof is essentially based on \cite[Theorem 1.15]{SeSi3} which was new as well, 
interesting on its own and, in the present context, reads as follows.
\begin{prop}\label{PComm}
Let $A\in\pb{\Xb}$. Then there is a uniform approximate projection 
$(R_n)\subset\pb{\Xb}$ with
\begin{itemize}
\item For every $m\in\Nb$ there is an $N_m\in\Nb$ such that for all $n\geq N_m$
			\[R_nP_m=P_mR_n=P_m \quad\text{and}\quad R_mP_n=P_nR_m=R_m.\]
\item The commutators $\|[A,R_n]\|:=\|AR_n-R_nA\|$ tend to zero as $n\to\infty$.
\end{itemize}
\end{prop}

\subparagraph{Fredholm operators and \texorpdfstring{$\Pc$}{P}-compact projections}
There is a third and very fruitful way of characterizing Fredholm operators:
One can ``capture'' their kernel and range via compact projections (As usual, 
an operator $P\in\lb{\Xb}$ is called projection if $P^2=P$):

\begin{itemize}
\item An operator $A\in\lb{\Xb}$ is Fredholm if and only if there exist projections
	 			$P,P' \in \lc{\Xb}$ such that $\im P = \ker A$ and $\ker P' = \im A$.
	 			
\item An operator $A\in\lb{\Xb}$ is not Fredholm if and only if for every $\epsilon>0$ and
				every $l\in\Nb$ there exists a projection $Q\in\lc{\Xb}$ 	 with $\rk Q \geq l$ 
				such that $\|AQ\|<\epsilon$ or $\|QA\|<\epsilon.$
\end{itemize}
This characterization is closely related to the existence of a generalized inverse $B$
for an operator $A$, that is $A=ABA$ and $B=BAB$ holds and $P=I-BA$, $P'=I-AB$.
\footnote{We will see some more details in the proof of Corollary \ref{CFredh}.}

Clearly, the question standing to reason is, whether the characterization of the
Fredholm property for operators in $\pb{\Xb}$ might be even possible with 
\mbox{$\Pc$-compact} projections instead of compact ones. Also this hope comes true 
as it was recently proved. Here are the details:

\begin{thm}\label{TPDich}(\cite[Proposition 1.27]{SeSi3})
Let $A\in\pb{\Xb}$.
\begin{itemize}
\item $A$ is Fredholm if and only if there exist projections
	 			$P,P' \in \pc{\Xb}$ of finite rank such that 
				$\im P = \ker A$ and $\ker P' = \im A$.
\item $A$ is not Fredholm if and only if for every $\epsilon>0$ and
				every $l\in\Nb$ there exists a projection $Q\in\pc{\Xb}$ with $\rk Q \geq l$ 
				such that $\|AQ\|<\epsilon$ or $\|QA\|<\epsilon.$ 
\end{itemize}
\end{thm}
Hence we can embed the classical Fredholm property into the $\Pc$-framework:
The implication ($A\in\pb{\Xb}$ $\Pc$-Fredholm $\Rightarrow$ $A$ Fredholm) holds 
for every $A\in\pb{\Xb}$ if and only if $\dim X<\infty$. This easily follows 
since $\pc{\Xb}\subset \lc{\Xb}$ if $\dim X<\infty$, whereas in case
$\dim X=\infty$ the projections $P_m$ are not compact and hence all $Q_m$ are
$\Pc$-Fredholm but not Fredholm. The converse implication (which is much more 
important since it guarantees that we can study all Fredholm operators by the 
tools that emerge from the $\Pc$-theory) always holds:
\begin{cor}\label{CFredh} 
Let $A\in\pb{\Xb}$. If $A$ is Fredholm then $A$ is $\Pc$-Fredholm and has a 
generalized inverse $B\in\pb{\Xb}$, i.e. $A=ABA$ and $B=BAB$. Moreover, $A$ is 
Fredholm of index zero if and only if there exists an invertible operator 
$C\in\pb{\Xb}$ and an operator $K\in\pc{\Xb}$ of finite rank such that $A=C+K$.
\end{cor}
\begin{proof}
Let $A\in\pb{\Xb}$ be Fredholm and $P,P' \in \pc{\Xb}$ denote projections as 
given by the previous theorem. The compression $A:\ker P\to\ker P'$ is an isomorphism
between Banach spaces, hence it has a bounded inverse $A^{(-1)}$ by the Banach Inverse 
Mapping Theorem. Now, the operator $B:=(I-P)A^{(-1)}(I-P')$ is a bounded linear 
operator on $\Xb$ with
\[AB=A(I-P)A^{(-1)}(I-P')=I-P' \text{ and } BA=(I-P)A^{(-1)}(I-P')A=I-P.\]
Thus, $B$ is a $\Pc$-regularizer for $A$ and Theorem \ref{TPFredh} yields the 
$\Pc$-Fredholm property. 
Clearly, 
\[ABA=(I-P')A=A-P'A=A \quad \text{and}\quad BAB=(I-P)B=B.\]
If, additionally, $\ind A=0$ then the projections $P,P'$ are of the same (finite) rank 
and we can choose a linear bijection $T:\im P \to \im P'$. We easily check that 
$K:=P'TP$ is $\Pc$-compact of the same finite rank: for this we just note that
\[\|Q_nK\|\leq\|Q_nP'\|\|TP\|\quad\text{and}\quad\|KQ_n\|\leq\|P'T\|\|PQ_n\|\]
tend to zero as $n$ goes to infinity. Now $C:=A+K$ is the desired invertible
operator with the inverse $C^{-1}=(I-P)A^{(-1)}(I-P')+PT^{-1}P'$, where $T^{-1}$ 
denotes the inverse of $T:\im P \to \im P'$.
Conversely, if $A$ arises from an invertible operator by a finite rank perturbation 
then it is Fredholm and $\ind A=0$.
\end{proof}

The next observation from \cite{SeSi3} which is a consequence of Theorem \ref{TPDich}
clarifies the Open Problem No. 4 in \cite{LiChW} and therefore constitutes the 
essential ingredient to obviate \eqref{CPrae}, the need for a predual setting, in the 
whole theory for the $l^\infty$-case:
\begin{prop}(\cite[Proposition 1.18 and Corollary 1.19]{SeSi3})\label{Prestrict}\\
Let $A\in\pb{l^\infty}$. Then $A$ maps the subspace $l^0$ into $l^0$ and its
compression $A|_{l^0}:l^0\to l^0$ satisfies $\|A|_{l^0}\|=\|A\|$. 
Moreover, $A$ is Fredholm if and only if $A|_{l^0}$ is Fredholm, and in this case 
\[\dim\ker A = \dim\ker A|_{l^0},\quad \dim\coker A = \dim\coker A|_{l^0},
 \quad \ind A = \ind A|_{l^0}.\]
\end{prop}

\subsection{The operator spectrum}

The designation ``invertibility at infinity'' already suggests that operators with
this property act on sequences $(x_i)\in\Xb$ which are supported ``far away from 
the origin'' similar to invertible operators. This is substantiated by the fact 
that invertibility at infinity aka $\Pc$-Fredholmness means invertibility up to
some $\Pc$-compact perturbations which have by definition their range of influence
essentially in a neighborhood of the origin. The next definition lays the ground 
for the limit operator technique which brings these observations into precise 
statements.

\begin{defn}
Let $A\in\pb{\Xb}$ and $(g_n)\subset\Zb^N$ be a sequence such that $|g_n|\to\infty$ 
as $n\to\infty$. The limit
\[A_g:=\plimn V_{-g_n} A V_{g_n}\]
of the sequence $(V_{-g_n} A V_{g_n})$ of shifted copies of $A$, if it exists,
is called the limit operator of $A$ with respect to the sequence $g$. The set 
$\sigma_{\op}(A)$ of all limit operators of $A$ is referred to as its operator spectrum.

Further, we say that $A$ is a rich operator if every sequence $g\subset\Zb^N$ of points 
whose absolute values tend to infinity has a subsequence $h$ such that $A_h$ exists.
The set of all rich operators is denoted by $\Lc^\$(\Xb,\Pc)$.
\end{defn}

The method of limit operators has been intensively studied during the last years and
proved to be an extremely useful tool with a wide range of applications in the theory of 
band-dominaed operators. Here we want to state only one of its highlights. More details
will follow in Section \ref{SBDO}.

\begin{thm}\label{TBdORichHS}
Let $A$ be a rich band-dominated operator. Then $A$ is $\Pc$-Fredholm if and only if
all limit operators of $A$ are invertible and their inverses are uniformly bounded.
\end{thm}

This theorem has been the engine for the development of the limit operator method, in a 
sense, and it has a long history from which we particularly mention the pioneering paper
\cite{LangeR} of Lange and Rabinovich. 
The proof of the \emph{if} part is based on a construction of a $\Pc$-regularizer, which 
has its roots in an idea of Simonenko \cite{Sim} and can be found in 
\cite{RochFredTh} and \cite{Marko}, for example. The \emph{only if} part was discussed
in \cite{RochFredTh} and \cite[Theorem 2.2.1]{LimOps} (for $1<p<\infty$), in \cite{Marko} 
(all $p$ and with the additional assumption \eqref{CPrae} on the existence of a predual 
setting in the case $p=\infty$), and in \cite{LiChW}, Theorem 6.28 (all $p$).
Our present approach and the new results of Theorems \ref{TPFredh} and \ref{TPDich} 
provide this implication already on the much more general level $\pb{\Xb}$: 

\begin{thm}\label{TLimOps}
Let $A\in\pb{\Xb}$ be $\Pc$-Fredholm. Then all limit operators 
of $A$ are invertible and their inverses are uniformly bounded.
Moreover, the operator spectrum of every $\Pc$-regularizer $B$ equals
\begin{equation}\label{ERegOpSpec}
\sigma_{\op}(B) =(\sigma_{\op}(A))^{-1} := \{A_g^{-1}:A_g\in\sigma_{\op}(A)\}.
\end{equation}
\end{thm}
\begin{proof}
Let $B$ be a $\Pc$-regularizer for $A$ and let $g=(g_n)$ be such that $A_g$ exists. It is 
quite obvious from the definition that the operator spectrum of $\Pc$-compact 
operators $K$ is trivial: $\sigma_{\op}(K)=\{0\}$. Thus, besides $V_{-g_n}AV_{g_n}\to A_g$, 
we also have $V_{-g_n}(AB-I)V_{g_n}\to 0$ and $V_{-g_n}(BA-I)V_{g_n}\to 0$ $\Pc$-strongly.
Then, for every $T\in\pc{\Xb}$,
\begin{align*}
\|T\| & = \|V_{-g_n}IV_{g_n}T\| 
			 \leq \|V_{-g_n}BV_{g_n}\|\|V_{-g_n}AV_{g_n}T\|
																				+ \|V_{-g_n}(I-BA)V_{g_n}T\|.
\end{align*}
Consequently (for $n\to\infty$ and with a constant $D>0$ independent of $g$ and $T$)
\[\|T\| \leq D\|A_gT\| \quad\text{for all}\quad T\in\pc{\Xb}.\]
The dual estimate $\|T\| \leq D\|TA_g\|$ for all $T\in\pc{\Xb}$ follows analogously.
Due to Theorem \ref{TPDich}, $A_g$ must be Fredholm with trivial kernel and cokernel,
hence invertible, and Theorem \ref{TPFredh} yields that $A_g^{-1}$ belongs to $\pb{\Xb}$.
Moreover,
\[V_{-g_n}BV_{g_n}-A_g^{-1} 
= V_{-g_n}BV_{g_n}(A_g-V_{-g_n}AV_{g_n})A_g^{-1} + V_{-g_n}(BA-I)V_{g_n}A_g^{-1},\]
hence $\|(V_{-g_n}BV_{g_n}-A_g^{-1})T\|\to 0$ as $n\to\infty$ for every $T\in\pc{\Xb}$.
Analogously we find $\|T(V_{-g_n}BV_{g_n}-A_g^{-1})\|\to 0$ and deduce that
$A_g^{-1}\in\sigma_{\op}(B)$. Thus the inclusion ``$\supset$'' in \eqref{ERegOpSpec}
is proved.

Interchanging the roles of $A$ and $B$, we can apply the above result to the 
$\Pc$-Fredholm operator $B\in\pb{\Xb}$ and its $\Pc$-regularizer $A$, find that every 
operator $B_g\in\sigma_{\op}(B)$ yields a limit operator $A_g=B_g^{-1}$ of $A$, and 
obtain the inclusion ``$\subset$'' in \eqref{ERegOpSpec}.

Finally, by Eq. \eqref{EPLim} in Proposition \ref{PPstrong}, the operators in 
$\sigma_{\op}(B)$ are uniformly bounded, which provides the uniform boundedness 
of the inverses $A_g^{-1}$.
\end{proof}

We point out the following nice properties of the set $\Lc^\$(\Xb,\Pc)$ of rich 
operators which generalize \cite[Proposition 1.2.8]{LimOps}.
\begin{cor}\label{CRichOps} We have
\begin{itemize}
\item The set $\Lc^\$(\Xb,\Pc)$ forms a closed subalgebra of $\pb{\Xb}$ and contains
			$\pc{\Xb}$ as a closed two-sided ideal.
\item Every $\Pc$-regularizer of a rich $\Pc$-Fredholm operator is rich. Thus, 
			$\Lc^\$(\Xb,\Pc)/\pc{\Xb}$ is inverse closed in $\pb{\Xb}/\pc{\Xb}$ and
			$\Lc^\$(\Xb,\Pc)$ is inverse closed in both $\pb{\Xb}$ and $\lb{\Xb}$.
\item For every $A\in\Lc^\$(\Xb,\Pc)$, $\Xb=l^p$, $p\in\{0\}\cup[1,\infty)$, 
			we have $A^*\in\Lc^\$(\Xb^*,\Pc)$ and 
			\begin{equation}
			\sigma_{\op}(A^*)=(\sigma_{\op}(A))^*:=\{A_g^*:A_g\in\sigma_{\op}(A)\}.
			\end{equation}
\end{itemize}
\end{cor}
\begin{proof}
The first part only requires some straightforward calculations, and the second
part is an immediate consequence of the previous proof: Every sequence $h$ has a 
subsequence $g$ such that $A_g$, and hence also $B_g=A_g^{-1}$, exists.

Now, let $A\in\Lc^\$(\Xb,\Pc)$. Then, by Proposition \ref{PLP}, $A^*\in\Lc(\Xb^*,\Pc)$ 
and $A_g\in\sigma_{\op}(A)$ always yields $(A_g)^*=(A^*)_g\in\sigma_{\op}(A^*)$. 
In particular, $A^*\in\Lc^\$(\Xb^*,\Pc)$.
Conversely, if $(A^*)_h\in\sigma_{\op}(A^*)$ then, due to the richness of $A$, there 
is a subsequence $g$ of $h$ such that $A_g$ exists, and then necessarily 
$(A_g)^*=(A^*)_g=(A^*)_h$. 
\end{proof}
Notice that the converse implication of Theorem \ref{TLimOps}, which is true for rich 
band-dominated operators by Theorem \ref{TBdORichHS}, does not hold in the general case:
\begin{ex}
Consider the space $l^2(\Zb,\Cb)$, the symbol function 
\[a:\Tb\to\Cb,\quad e^{\ii t}\mapsto -t/\pi+1,\quad t\in[0,2\pi),\] 
and the Toeplitz plus Hankel operator $A=I-\ii \chi_+H(a)\chi_+I$. 
From \cite[2.4.2 (3) and Example 2.2]{HaRoSi} we get that that this operator
is not Fredholm, hence not $\Pc$-Fredholm. However, it is obviously rich and its 
operator spectrum is the singleton $\{I\}$.
\end{ex}

\subparagraph{On the ``big question''}
The ``big question'' in the limit operator business as it is stated in 
\cite[Section 3.9]{Marko}, and also as the Open Problem No. 8 in \cite{LiChW}, is 
the following:
\begin{equation}\label{CBig}
\begin{minipage}{10cm}
Is the operator spectrum of a rich operator automatically uniformly invertible
if it is elementwise invertible?
\end{minipage}
\end{equation}
For this we say that the operator spectrum of an operator $A$ is \textit{elementwise 
invertible} if all limit operators of $A$ are invertible, and \textit{uniformly 
invertible} means that additionally the inverses are uniformly bounded.
Clearly, it would be a great improvement and simplification of Theorem \ref{TBdORichHS}
if this question could be answered affirmatively. 
Currently only the following partial answer is known:

\begin{thm}\label{TInvaIAlp}
Let $p\in\{0,1,\infty\}$ and $A\in\Lc^\$(l^p,\Pc)$. If $\sigma_{\op}(A)$ is 
elementwise invertible then it is uniformly invertible.
\end{thm}
\begin{proof}
Actually, this has been studied and proved for rich band-dominated 
operators so far. A comprehensive survey as well as the respective proof are given in 
\cite[Section 3.9]{Marko} and \cite[Theorem 6.28]{LiChW}. 

Let $p=\infty$. Now, having Proposition \ref{PComm} available and plugging in 
the operators $(R_n)$ which asymptotically commute with $A\in\pb{l^p}$, every step 
of the proof in \cite{Marko} works without taking any properties of band-dominated
operators into account. Besides that replacement there are no further modifications 
of that proof needed, so we omit to repeat its details again.
However, it should be emphasized that this affirmative answer is, in fact, not caused 
by the particular advantages of band-dominated operators, but only by the structure 
of $\Lc^\$(l^\infty,\Pc)$ on these particular spaces $l^\infty$.

If $A\in\Lc^\$(l^1,\Pc)$ has elementwise invertible operator spectrum then, by 
Corollary \ref{CRichOps}, its adjoint $A^*\in\Lc^\$(l^\infty,\Pc)$ has elementwise, 
hence even uniformly, invertible operator spectrum. Since $\|A_g^{-1}\|=\|(A_g^{-1})^*\|$ 
for every $A_g\in\sigma_{\op}(A)$, we get the claim for $p=1$.
The case $l^0$ is treated analogously, using its duality to $l^1$.
\end{proof}

For the cases $p\in(1,\infty)$ this main problem is still open in general, and could 
be answered affirmatively only for some 
special classes of operators, such as those in the Wiener algebra or for 
band-dominated operators with slowly oscillating coefficients (see \cite{LimOps, Marko} 
and below).

\subsection{A first roundup}\label{Sroundup}
We recall the following collection of conditions from \cite[(5.14)]{LiChW}:
\begin{align*}
{(a)}\quad& A\text{ is invertible}\\
{(b)}\quad& A\text{ is Fredholm}\\
{(c)}\quad& A\text{ is $\Pc$-Fredholm}\\
{(d)}\quad& \sigma_{\op}(A)\text{ is uniformly invertible}\\
{(e)}\quad& \sigma_{\op}(A)\text{ is elementwise invertible}\\
{(f)}\quad& \text{All limit operators of $A$ are injective.}
\end{align*}
We have seen that for $A\in\pb{\Xb}$ all implications
$(a)\Rightarrow(b)\Rightarrow(c)\Rightarrow(d)\Rightarrow(e)\Rightarrow(f)$ hold.
Concerning the converse implications we have the following: $(c)\Rightarrow(b)$ holds 
in case of compact $\Pc=(P_n)$, $(d)\Rightarrow(c)$ is the striking advantage of rich
band-dominated operators, which we will study in more detail within the next section.
The implication $(e)\Rightarrow(d)$ is the big question, and particularly true for 
rich operators in the extremal cases $p\in\{0,1,\infty\}$, or for all rich operators in 
the Wiener algebra. The latter will be subject of Section \ref{SWiener}. 
Finally, also $(f)\Rightarrow(e)$ holds true in a surprisingly wide range of cases. 
We will focus on that in Section \ref{SFavard}.

\section{Fredholm theory of band-dominated operators}\label{SBDO}

\subsection{General band-dominated operators}

Let us start this section with mentioning again that all band-dominated operators 
belong to the algebra $\pb{\Xb}$. This is obvious for shifts and operators of 
multiplication and then easily follows for all $A\in\Alp$ since $\pb{\Xb}$ is a 
Banach algebra.

Therefore, the results presented in the previous section suggest to study the
Fredholm property of band-dominated operators with the help of limit operators.
Actually, in the beginning most of the above results had been studied for 
band-dominated operators only. Later on, with the wisdom of hindsight, it turned out 
that large parts of the theory are not specific for $\Alp$ but hold in $\pb{\Xb}$. 
Of course, one striking advantage of rich band-dominated operators is already outlined 
in Theorem \ref{TBdORichHS}, the equivalence of $\Pc$-Fredholmness and the uniform 
invertibility of the operator spectrum. 
Another remarkable speciality in the case $\dim X < \infty$
is the fact that \textit{every} band-dominated operator turns out to be rich 
\cite[Corollary 2.1.17]{LimOps}. The following collection of results shall underline 
that $\Alp$ is self-contained and closed with respect to the terms of the $\Pc$-theory.

\begin{thm} \label{TBDO}
(\cite[Prop 2.1.7. et seq.]{LimOps} or \cite[Propositions 1.46, 2.9, 2.11, 3.6]{Marko})\\
For $\Alp$ and $A\in\Alp$ we have the following.
\begin{itemize}
\item The algebra $\Alp$ contains $\pc{\Xb}$ as a closed two-sided ideal. 
\item The limit operators of band-dominated operators always 
			belong to $\Alp$ again, i.e. $\sigma_{\op}(A)\subset\Alp$.
\item If $A$ is invertible then $A^{-1}\in\Alp$. Thus $\Alp$ 
			is inverse closed in $\lb{\Xb}$ and $\pb{\Xb}$.
\item If $A$ is $\Pc$-Fredholm then every $\Pc$-regularizer of $A$ is 
			band-dominated as well.
\end{itemize}
\end{thm}

The following observations which additionally address the usual Fredholm property are 
probably new.
\begin{thm}\label{TBDO2}
Let $A\in\Alp$.
\begin{itemize}
\item If $A$ is Fredholm then there exists a regularizer which belongs to $\Alp$.
			More precisely, there is even a generalized inverse $B\in\Alp$ for $A$,
			which means that $ABA=A$, $BAB=B$ holds and $I-AB$, $I-BA$ are both compact and 
			$\Pc$-compact projections. 
\item If $A$ is invertible, $\Pc$-Fredholm or Fredholm and $B$ is an inverse, 
			$\Pc$-regularizer or Fredholm regularizer in $\Alp$ for $A$, respectively, 
			then $\sigma_{\op}(B)=(\sigma_{\op}(A))^{-1}$. Moreover, $A$ 
			is rich if and only if $B$ is rich.
\end{itemize}
\end{thm}
\begin{proof}
The first assertion is a consequence of Corollary \ref{CFredh} together with Theorem
\ref{TBDO}. The second assertion follows from Theorem \ref{TLimOps} and Corollary \ref{CRichOps}.
\end{proof}

\subparagraph{The case \texorpdfstring{$N=1$}{N=1} and the Fredholm index}
We now suppose that $N=1$ and our aim is to state index formulas for 
operators in terms of their limit operators.

We already know by Corollary \ref{CFredh} that a Fredholm operator $A\in\Alp$ has 
the $\Pc$-Fredholm property and hence, due to Theorem \ref{TLimOps}, its limit 
operators are invertible.  Rabinovich, Roch and Roe \cite{RochFredInd} proposed to 
consider the following compressions of the limit operators:
Let $P$ stand for the projection $P:=P_{\Nb}$ and set $Q:=I-P$. For a given operator
$A\in\Alp$ define 
\[A_+:=PAP+Q \quad\text{and}\quad A_-:=QAQ+P\]
and further introduce the plus-(and minus-)index of $A$ as $\ind^{\pm} A := \ind A_{\pm}$.
Finally let $\sigma_{\pm}(A)$ denote the set of all limit operators $A_g$ of $A$ w.r.t.
sequences $g$ tending to $\pm\infty$, respectively. Then the main observations of 
\cite{RochFredInd} for Fredholm operators $A\in\Alp$ in the case $\dim X < \infty$ and
$p=2$ have been the formulas 
\[\ind A=\ind_+A + \ind_-A\quad\text{and}\quad \ind_\pm A=\ind_\pm A_g\] 
for arbitrary limit operators $A_g\in\sigma_\pm(A)$, respectively.
This result was generalized to the case $\dim X < \infty$ and $p\in(1,\infty)$ in
\cite{Rochlp}, and later on in \cite{LpInd} to Banach spaces $X$ with a certain 
additional property, the \textit{symmetric approximation property}, and for a 
particular class of operators of the form $A=I+K$ with rich $K\in\Alp$ having 
compact matrix entries. The most general version of such an index formula, without 
any restrictions on $X$ and for arbitrary band-dominated operators, has been proved 
in \cite[Theorem 3.7]{SeSi3} and, in the present notation, reads as follows:
\begin{thm}\label{TIndex}
Let $A\in\Alp$ be Fredholm. Further, let $(u_n), (l_n)\subset \Zb$ be sequences 
tending monotonically to $+\infty$ or $-\infty$, respectively, such that 
$A_u\in\sigma_+(A)$ and $A_l\in\sigma_-(A)$ exist. Finally set 
$L_n:=P_{\{l_n,\ldots,u_n\}}$. If $(A_u)_+$ and $(A_l)_-$ are Fredholm operators 
then there is a number $n_0\in\Nb$ such that for $n\geq n_0$ the operators 
$L_nAL_n \in\lb{\im L_n}$ are Fredholm as well and 
\begin{equation}\label{EIndExt}
\ind A = \ind L_nAL_n + \ind_+(A_u) + \ind_-(A_l).
\end{equation}
\end{thm}
Notice that in all cases which have been studied before (i.e. for $\dim X<\infty$ or 
the operators with compact entries of \cite{LpInd}), the Fredholm property of $A$ 
automatically implies the Fredholm property of $(A_u)_+$ and $(A_l)_-$, and the 
compressions $L_nAL_n$ always have index zero. This again leads to the simplified 
formula
\begin{equation}\label{EInd}
\ind A = \ind_+(A_u) + \ind_-(A_l).
\end{equation}
It should also be mentioned that Theorem \ref{TIndex} even holds for the more 
general quasi-banded operators which will be discussed in our final Section \ref{SGen}.

\subsection{Operators in the Wiener algebra}\label{SWiener}

Recall the Wiener algebra $\Wc$ which is defined as the closure of the algebra $\Bc$ 
of all band operators $A=\sum_\alpha a_\alpha V_\alpha$ with respect to the norm
\[\|A\|_{\Wc}:=\sum_\alpha\|a_\alpha\|_\infty.\]
Before we come to the proof of Theorem \ref{TWiener2} let us shortly summarize
to what extend the previous observations for $\Alp$ specify to $\Wc$. 
From \cite[Theorem 2.5.2]{LimOps} we know that, for every $p\in\{0\}\cup [1,\infty]$, 
$\Wc$ is inverse closed in $\lb{l^p}$, hence also in $\pb{l^p}$ and $\Alp$. Moreover, 
by \cite[Propositions 2.5.1 and 2.5.6]{LimOps}, the property of the operators 
$A\in\Wc$ to be rich does not depend on $p$, the operator spectrum of rich operators 
in the Wiener algebra is independent from $p$ as well, and it is always a subset of 
$\Wc$. This leads to the great advantage of this class that one can extend the  
affirmative answer to the big question, via interpolation, to every $p$ 
(cf. \cite{LangeR}, \cite[Theorem 2.5.7]{LimOps} and \cite[Theorem 6.40]{LiChW}): 
\begin{thm}\label{TPWiener}
Let $A\in\Wc$ be rich. Then the following are equivalent
\begin{enumerate}
\item $A$ is $\Pc$-Fredholm on one of the spaces $l^p$.
\item $A$ is $\Pc$-Fredholm on all the spaces $l^p$.
\item All limit operators of $A$ are invertible on one of the spaces $l^p$.
\item All limit operators of $A$ are invertible on all the spaces $l^p$ and
 			\begin{equation}\label{EUniBdd}
 			\sup_{p\in \{0\}\cup[1,\infty]}\sup_{A_g\in\sigma_{\op}(A)}\|A_g^{-1}\|_{\lb{l^p}}< \infty.
 			\end{equation}
\end{enumerate}
\end{thm}
\begin{proof}
The implications 
$\textit{4.}\Rightarrow \textit{2.}\Rightarrow \textit{1.}\Rightarrow \textit{3.}$ 
are clear by Theorems \ref{TBdORichHS} and \ref{TLimOps}. So let $\textit{3.}$ 
hold true. Since the operator spectrum of $A\in\Wc$ does not depend on $p$ and 
since $\Wc$ is inverse closed on every $l^p$, we see that all limit operators of 
$A$ are invertible on all the spaces $l^p$. For $p\in\{1,\infty\}$ there exists 
a uniform bound by Theorem \ref{TInvaIAlp}. Finally, the Riesz-Thorin Interpolation 
Theorem and Proposition \ref{Prestrict} provide
\[\|A_g^{-1}\|_p\leq\max\{\|A_g^{-1}\|_1,\|A_g^{-1}\|_\infty\}\]
for all $p\in\{0\}\cup [1,\infty]$ and all $A_g\in\sigma_{\op}(A)$. Thus 
\eqref{EUniBdd} holds.
\end{proof}

Let us now again turn our attention to the (classical) Fredholm property and the 
proof of Theorem \ref{TWiener2}. 
In fact, our arguments follow in large parts those of the original ones 
\cite{MarkoWiener}, just profit from the new and improved results of Section \ref{SAP}, 
and use an alternative construction of Fredholm operators of prescribed index that 
does not require the hyperplane property of $X$. We start with an auxiliary lemma.

\begin{lem}\label{LHyper}
For every $\kappa\in\Zb$ there exists an operator $S_\kappa\in\Wc$ which is Fredholm
on $l^p=l^p(\Zb^N,X)$ of index $\kappa$ for every $p\in\{0\}\cup[1,\infty]$.
\end{lem}
\begin{proof}
Choose a projection $\Rh\in\lb{X}$ of rank $1$ (see e.g. \cite[B.4.9]{Pietsch}) 
and define a projection $R:l^p\to l^p$, $(x_i)\mapsto (\Rh x_i)$. Moreover 
introduce the projections $\hat{P}:=P_{\Nb\times\{0\}^{N-1}}$ and $\hat{Q}:=I-\hat{P}$. 
Now set 
\[S_{|\kappa|}:=\hat{Q}+\hat{P}(I-R+V_{(-|\kappa|,0,\ldots,0)}R)\hat{P} \quad\text{and} \quad
 S_{-|\kappa|}:=\hat{Q}+\hat{P}(I-R+V_{(|\kappa|,0,\ldots,0)}R)\hat{P}\]
and easily check that $S_{|\kappa|}S_{-|\kappa|}=I$, whereas the spaces
$\ker S_{|\kappa|}=\im RP_{\{0,\ldots,|\kappa|-1\}\times\{0\}^{N-1}}$ as well as 
$\im S_{-|\kappa|}=\ker RP_{\{0,\ldots,|\kappa|-1\}\times\{0\}^{N-1}}$ are of the
dimension (resp. codimension) $|\kappa|$, independent from the choice of $p$.
This easily shows that $S_{|\kappa|}$ and $S_{-|\kappa|}$ are one-sided invertible 
Fredholm operators in $\Wc$ of the index $|\kappa|$ and $-|\kappa|$, respectively.
To convey a better understanding of how these operators act, we mention that for 
$N=1$ the matrix representation of $S_{|\kappa|}$ is given by
\[[S_{|\kappa|}] = \begin{pmatrix}
\ddots\hspace{15pt}& & & & & & & & & \\
& I\hspace{15pt} & & & & & & & & \\
& & I\hspace{15pt} & & & & & & & \\
& & & I-\Rh & 0 & \hspace{-20pt}\stackrel{\text{$|\kappa|-1$ times}}{\ldots}\hspace{-20pt} & 0 & \Rh & & \\
& & & & I-\Rh & 0 &   & 0 & \Rh & \\
& & & & & \ddots\hspace{15pt} & \ddots & & \ddots & \ddots \\
\end{pmatrix}.\]
\end{proof}

This particularly shows that the sequence spaces $l^p$ have the hyperplane
property, independently from the geometry of $X$. The proof of Theorem \ref{TWiener2}
is now straightforward:

\begin{proof}
Let $A\in\Wc$ be Fredholm on one of the spaces $l^p$ and $\kappa:=\ind A$. Then, with
$S_{-\kappa}$ as in Lemma \ref{LHyper}, $AS_{-\kappa}$ is Fredholm of index zero, and 
Corollary \ref{CFredh} provides an operator $K\in\pc{l^p}$ of finite rank $d$ such that 
$AS_{-\kappa}-K$ is invertible on $l^p$. From the definitions, we easily derive that
the operators $P_mKP_m$ belong to $\Wc\cap\pc{l^p}$, have finite rank less than or 
equal to $d$ and converge to $K$ in the operator norm on $l^p$ as $m\to\infty$. 
Since the set of invertible operators is open, we get that for a sufficiently large 
$m$ also $B:= AS_{-\kappa}-P_mKP_m$ is invertible and, moreover, belongs to $\Wc$. 
Due to the inverse closedness of $\Wc$, its inverse $B^{-1}$ is in the Wiener algebra 
as well. Thus, we have
\[I=BB^{-1}= AS_{-\kappa}B^{-1}-P_mKP_mB^{-1}=AC-T\]
with $C:=S_{-\kappa}B^{-1} \in\Wc$ and $T:=P_mKP_mB^{-1}\in\Wc\cap\pc{l^p}$ where
$\rk T \leq d$. Since $P_mKP_m$ considered as operator on $l^r$, 
$r\in\{0\}\cup[1,\infty]$, vanishes on $\im Q_m$ and has a range being a finite 
dimensional subspace of $\im P_m$ we see that $T$ is an operator of finite rank $d$
on $l^r$ as well. Proceeding in a completely symmetric way one also arrives at an  
equation $I=C'A-T'$ with $C'\in\Wc$ and $T'\in\Wc\cap\pc{l^r}$ of finite rank.
Since finite rank operators are compact, we see that $A+\lc{l^r}$ is invertible
in the Calkin algebra $\lb{l^r}/\lc{l^r}$, hence $A$ is Fredholm on $l^r$.
From the equalities
\[0=\ind(I)=\ind(AC-T)=\ind(AS_{-\kappa}B^{-1})= \ind A + \ind S_{-\kappa}\]
we derive the desired relation $\ind A = \kappa$ on $l^r$.
Actually,
\[C=IC=(C'A-T')C=C'(AC-T)+C'T-T'C = C' + (C'T-T'C),\]
hence the difference $C-C'$ is of finite rank, and we can even conclude that both
operators $C,C'$ are Fredholm regularizers for $A$ and belong to $\Wc$.
\end{proof}
Actually, this proof does not only provide that Fredholmness of operators in the 
Wiener algebra and their indices are independent of the underlying space, but it 
also yields regularizers in $\Wc$. Therefore we can specify Theorem \ref{TBDO2}
as follows.
\begin{cor}\label{CWiener}
Let $A\in\Wc$ be Fredholm. Then there exists an operator $B\in \Wc$ which is a
Fredholm regularizer for $A$ on every space $l^p$.
In this case the operator spectrum $\sigma_{\op}(B)$ of $B$ coincides with 
$\{A_g^{-1}:A_g\in\sigma_{\op}(A)\}$. Moreover, $A$ is rich if and only if $B$ is 
rich.
\end{cor}
\subsection{Collective compactness and Favard's condition}\label{SFavard}

Chandler-Wilde and Lindner \cite{LiChWFavard} have studied another more concrete class 
of band-dominated operators on $l^\infty(\Zb,X)$ for which the Fredholm criteria become 
much simpler in the sense that already the injectivity of all of their limit operators 
is sufficient for the Fredholmness of such operators.

\begin{defn}
We say that an operator $A$ on $l^\infty(\Zb,X)$ is subject to Favard's condition if
every limit operator of $A$ is injective on $l^\infty(\Zb,X)$.
\end{defn}

Let us first introduce the precise framework.
\begin{defn}
Let $\Uc\Mc$ denote the set of all operators $K\in\Ac_{\infty}$ with the property that 
the operators $k_{ij}:=P_{0}V_{-i}KV_{j}P_{0}$ ($i,j\in\Zb$) form a collectively compact
set, that is
\[\{k_{ij}x:i,j\in\Zb^N,x\in l^\infty(\Zb,X), \|x\|\leq 1\}\text{ is relatively compact}.\]
\end{defn}
If one interprets $K$ as an infinite matrix which acts on the sequence space 
$l^\infty(\Zb,X)$ then the operators $k_{ij}$ can be regarded as its matrix entries.
\footnote{A rigorous description and justification of this perspective can be found in 
\cite{LimOps} or \cite{Marko}.}

Note that the set $\Uc\Mc$ is a Banach space, and its elements are locally compact
operators in the sense of \cite{LpInd}. Moreover, $\Uc\Mc$ is a Banach subalgebra and 
a left-sided ideal of $\Ac_{\infty}$. In case $\dim X<\infty$ it holds that 
$\Uc\Mc=\Ac_{\infty}$. These properties are proved among others in 
\cite[Section 2]{LiChWFavard} or \cite[Section 6.3]{LiChW}.
Here comes the striking advantage of this class as it appears in 
\cite[Theorem 3.1]{LiChWFavard} and \cite[Theorem 6.31]{LiChW}.

\begin{thm}\label{TUM}
Let $A=I-K$ with $K\in\Uc\Mc$ be rich and all limit operators be injective.
Then all limit operators of $A$ are invertible.
\end{thm}

\begin{cor}\label{CUM}
Let $A=I-K$ with $K\in\Uc\Mc$ be rich. Then the following are equivalent.
\begin{enumerate}
\item All limit operators of $A$ are injective (Favard's condition).
\item The operator spectrum of $A$ is uniformly invertible.
\item $A$ is $\Pc$-Fredholm.
\item $A$ is Fredholm.
\item There exists a rich $B\in\Ac_{\infty}$ s.t. $ABA=A$, $BAB=B$ holds and
			$I-BA$, $I-AB$ are compact projections onto the kernel or parallel to the
			range of $A$, respectively.
\end{enumerate}
\end{cor}
We want to emphasize again that this covers \textit{all} operators $A\in\Ac_\infty$ in 
case $\dim X<\infty$.
\begin{proof}
The implication \textit{1. $\Rightarrow$ 2.} is Theorem \ref{TUM} together with 
Theorem \ref{TInvaIAlp}, \textit{2. $\Rightarrow$ 1.} is obvious and 
\textit{2. $\Leftrightarrow$ 3.} 
follows from Theorem \ref{TBdORichHS}. Furthermore, \textit{4. $\Rightarrow$ 5.} holds 
by Theorem \ref{TBDO2}, \textit{5. $\Rightarrow$ 4.} is evident, and Corollary 
\ref{CFredh} provides the implication \textit{4. $\Rightarrow$ 3.} Therefore it remains 
to prove \textit{3. $\Rightarrow$ 4.}:
Let $B$ be a $\Pc$-regularizer and let $\Xb$ stand for $l^\infty(\Zb,X)$. 
Then $(I-K)B-I=AB-I=:T\in\pc{\Xb}$ hence $B=I+KB+T$. Thus
\begin{align*}
A(I+KB)&=(I-K)(I+KB)=I-K(I-B+KB)=I+KT
\end{align*}
and $A(I+KB)(I-KT)=I-KTKT$.
Since $T\in\pc{\Xb}$ it holds that $\|TP_mKP_mT-TKT\|\to 0$ as $m\to\infty$.
Furthermore, $P_mKP_m$ is compact, hence $TKT$ and even $KTKT$ are compact, 
which shows that $A+\lc{\Xb}$ has a right inverse in $\lb{\Xb}/\lc{\Xb}$. 
Analogously we find a left inverse for $A+\lc{\Xb}$, and we conclude that 
$A$ is Fredholm.
\end{proof}
This Corollary already appeared in \cite{LiChW} as Corollary 6.32 in large parts, but 
under the additional assumption \eqref{CPrae} on the existence of a predual setting, 
because at that time Corollary \ref{CFredh} and hence the implication 
\textit{4. $\Rightarrow$ 3.} were not available. 

The Open Problem No.~6 in \cite{LiChW} conjectures that Theorem \ref{TUM} may also 
hold in case $l^\infty(\Zb^N,X)$, $N>1$. Unfortunately, the following example 
demonstrates that this is wrong.

\begin{ex}
Let $X$ be a Banach space and $\Rh\in\lb{X}$ be a rank-one-projection.
Define projections $R$ and $M$ on $l^\infty(\Zb^2,X)$ by the rules
$R(x_i):=(\Rh x_i)$ and $M(x_i):=(y_i)$ with 
\[y_i:=\begin{cases} x_i&\text{ if } i=(k,l)\in\Zb^2\text{ with either }l\leq 0 
	\text{ or } k\geq l^2\\
 0&\text{ otherwise}. \end{cases}\]
Furthermore, set $K=(I - V_{(1,0)})MR$ and $A=I-K$.
Clearly, $R$, $M$, $K$ and $A$ are rich band operators and $K\in\Uc\Mc$.

We easily check that the operator spectrum of $M$ consists of 
the zero operator, the identity and all shifted copies $V_{-\alpha}M_kV_\alpha$
($\alpha\in\Zb^2$) of the operators $M_k$ ($k=1,2$) which are given
by $M_k(x_i):=(y_i^k)$ with  
\begin{align*}
 y_i^1&:=\begin{cases} x_i&\text{ if } i=(k,l)\text{ with }l\leq 0\\
 0&\text{ otherwise}, \end{cases}\\
 y_i^2&:=\begin{cases} x_i&\text{ if } i=(k,l)\text{ with }l< 0 \text{ or } k\geq l=0\\
 0&\text{ otherwise}. \end{cases}
\end{align*}
This yields that $\sigma_{\op}(A)$ consists of the identity, the operator
$(I-R)+V_{(1,0)}R$ and all shifted copies of the operators
$I-(I - V_{(1,0)})M_kR$ ($k=1,2$). With the representation
\[I-(I - V_{(1,0)})M_kR=(I-R) + R((I-M_k) + V_{(1,0)}M_k)R\] 
we finally check that all limit operators are injective, but 
$I-(I - V_{(1,0)})M_2R$ is not invertible.
\end{ex}

We finish this section with generalized versions of 
\cite[Proposition 4.1, and Corollary 4.3]{LiChWFavard}
which combine the specialities of both, the algebras $\Wc$ and $\Uc\Mc$, in order 
to extend the previous observations to all spaces $l^p$. Our present result is 
homogenous w.r.t. all $p\in\{0\}\cup[1,\infty]$ and there are not any restrictions 
on the Banach space $X$ remaining, in particular for the index formula.

\begin{cor}
Let $A=I-K$ with $K\in\Uc\Mc\cap\Wc$ be rich. The following are equivalent.
\begin{itemize}
\item[FC.] All limit operators of $A$ are injective on $l^\infty$ (Favard's condition).
\item[1.] All limit operators of $A$ are invertible on \textbf{one} of the spaces $l^p$.
\item[2.] $A$ is $\Pc$-Fredholm on \textbf{one} of the spaces $l^p$.
\item[3.] $A$ is Fredholm on \textbf{one} of the spaces $l^p$.
\item[4.] There exists a Fredholm regularizer $B\in\Wc$ for $A$ on \textbf{one} of the 
					spaces $l^p$.
\item[1'.-4'.] The same as in 1.-4. but with \textbf{one} replaced by \textbf{all}.
\end{itemize}

If $A$ is Fredholm then its index is the same on all the spaces $l^p$ and is given by
the index formula \eqref{EInd}, $\ind(A)=\ind_+(A_u) + \ind_-(A_l)$ with arbitrary 
$A_u\in\sigma_+(A)$ and $A_l\in\sigma_-(A)$. Furthermore, the regularizers and the  
operator spectra of such Fredholm operators $A$ are independent from $p$ and 
fulfill \eqref{EUniBdd} as well as \eqref{ERegOpSpec}.
\end{cor}
One may ask, if $l^\infty$ in the first statement \textit{FC} could even be replaced 
by $l^p$, but already the simple band operator $A=I-V_1$ dashes this hope, hence the 
first condition in the $l^\infty$-case is really stronger than in the $l^p$-case.
Indeed, $\ker A=\{(c):c\in\Cb\}\subset l^\infty$, and $\sigma_{\op} A=\{A\}$, since $A$ 
is shift invariant, hence all limit operators of $A$, regarded as operator on $l^p$ 
with $p<\infty$, are injective, but, regarded as operator on $l^\infty$, that is 
obviously not true. 

The reader is encouraged to compare these results with the observations of Section
\ref{Sroundup}.
We also note again that, in the case $\dim X<\infty$, $\Uc\Mc=\Ac_{\infty}$ and every
band-dominated operator is rich, thus this corollary applies to all $A\in\Wc$
(cf. also \cite[Section 6.5]{LiChW}).
\begin{proof}
We additionally introduce the conditions
\begin{itemize}
\item[\textit{5.}] \textit{$A$ is $\Pc$-Fredholm on $l^\infty$.}
\item[\textit{6.}] \textit{$A$ is Fredholm on $l^\infty$.}
\end{itemize}
$\textit{FC.} \Leftrightarrow \textit{5.}\Leftrightarrow \textit{6.}$ is Corollary 
\ref{CUM}. $\textit{5.}\Leftrightarrow \textit{1.'}\Leftrightarrow \textit{1.}
\Leftrightarrow \textit{2.'}\Leftrightarrow \textit{2.}$ and 
Equation \eqref{EUniBdd} are provided by Theorem \ref{TPWiener}. Equation
\eqref{ERegOpSpec} is proved in Theorem \ref{TLimOps}.
$\textit{6.}\Leftrightarrow \textit{3.'}\Leftrightarrow \textit{3.}$ as well as the 
independence of the index follow from Theorem \ref{TWiener2}. Furthermore, 
Corollary \ref{CWiener} yields $\textit{3'.}\Rightarrow \textit{4.'}$ whereas 
$\textit{4'.}\Rightarrow \textit{4.}\Rightarrow \textit{3.}$ are evident.

So, let $A=I-K$ be Fredholm. The operator $K\in\Wc$ is the norm limit of a sequence 
$(K^{(m)})$ of operators $K^{(m)}:=\sum_{\alpha=1}^m a_\alpha V_\alpha$, where every 
entry of every diagonal $a_\alpha$ is a compact operator on $X$. Consequently, every limit
operator $K_g$ of $K$ is the norm limit of the respective limit operators $K_g^{(m)}$.
The latter are still finite sums of the form $\sum_{\alpha=1}^m b^{(m)}_\alpha V_\alpha$
having compact entries. $PK_g^{(m)}Q$ (analogously $QK_g^{(m)}P$) are 
\[PK_g^{(m)}Q=P\sum_{\alpha=1}^m b^{(m)}_\alpha V_\alpha Q
=P\sum_{\alpha=1}^m \sum_{\beta\in\Zb}b^{(m)}_\alpha(\beta)P_{\{\beta\}} V_\alpha Q
= P\sum_{\alpha=1}^m \sum_{\beta=1}^\alpha b^{(m)}_\alpha(\beta)P_{\{\beta\}}) V_\alpha Q, \]
i.e. finite sums of compact operators. Therefore $PK_gQ$ and $QK_gP$ are compact, 
hence $PA_gP+QA_gQ$ are compactly perturbed copies of the invertible operators $A_g$. 
So we see that all restricted limit operators $(A_u)_+$ and $(A_l)_-$ of $A$ as 
they appear in Theorem \ref{TIndex} are Fredholm, and we obtain the index formula 
\eqref{EIndExt} which even simplifies to \eqref{EInd} since $L_nKL_n$ is compact 
due to the above reasons.
\end{proof}

\section{Extensions and Generalizations}\label{SGen}

\subparagraph{One-sided definitions}
In \cite{LiChW} it is also discussed whether one can weaken and replace 
the definitions of $\pc{\Xb}$, $\pb{\Xb}$ and $\Pc$-strong convergence by one sided 
analogues (see the Open Problems No. 1 to 3 there). More precisely, one may define 
(\cite[Lemma 3.3, Corollary 3.5]{LiChW})
\begin{align*}
SN(\Xb):=&\{K\in\lb{\Xb}:\|KQ_n\|=\|KP_n-K\|\to 0 \text{ as } n\to\infty\}\\
S(\Xb):=&\{A\in\lb{\Xb}: KA\in SN(\Xb) \text{ for every }K\in SN(\Xb)\}\\
			=&\{A\in\lb{\Xb}:\|P_mAQ_n\|\to 0 \text{ as } n\to\infty \text{ for every }m\}
\end{align*}
and say that $(A_n)\subset\lb{\Xb}$ $s$-converges to $A$ (or has the $s$-limit $A$) if 
\[\|K(A_n-A)\|\to 0 \text{ for every }K\in SN(\Xb).\]
Let us consider the following example which illustrates that weakening the definition 
of $\pc{\Xb}$ necessitates also the modification of $\pb{\Xb}$ and the notion of
convergence:
\begin{ex}\label{EExpand}
Let $\psi:L^p(0,1)\to L^p(\Rb)$ be an isometric Banach space isomorphism. As an example
one may take the following construction: The mapping
\[\phi:L^p(\Rb_+)\to L^p(\Rb),\quad (\phi f)(t):=e^{-t/p}f(e^{-t})\]
is an isometry, and its inverse mapping is 
$(\phi^{-1}g)(s)=s^{-1/p} g(-\log s)$, $s>0$. The same holds true for the restriction
$\phi:L^p(0,1)\to L^p(\Rb_+)$. Thus $\psi:=\phi\circ\phi:L^p(0,1)\to L^p(\Rb)$ does 
the job. Now, let $\Xb:=l^p(\Zb,L^p(0,1))$ and define $B\in\lb{\Xb}$ by
\[(B(x_n))_k:=\chi_{(0,1)}V_{-k}\psi x_0,\quad k\in\Zb.\]
Clearly $B$ is an isometric isomorphism between $\im P_0$ and $\Xb$, $B$ belongs to 
$SN(\Xb)$, but $\|Q_nBP_0\|$ does not tend to zero as $n\to\infty$. Thus, we see that 
$SN(\Xb)$ is not a subset of $\pb{\Xb}$, hence would not serve as an ideal there.
Moreover, $\|(A_n-A)B\|\to 0$ would always imply $\|A_n-A\|\to 0$, so a two-sided 
definition of convergence based on $SN(\Xb)$ instead of $\pc{\Xb}$ would be nothing 
but the usual norm convergence.
\end{ex}

The motivation for such modified definitions is obvious: One can now try to develop an 
analogous theory for the larger family $S(\Xb)\supset\pb{\Xb}$. To be more concrete, one 
may ask which of the results of Section \ref{SAP} translate to this more general setting 
and, in particular, whether there is an analogous Fredholm and limit operator theory 
available.

An initial dawn of hope arises with the observation that Propositions \ref{PLP} and 
\ref{PPstrong} are valid in the present setting as well: 

\begin{prop}
Let $\Xb$ be a Banach space with uniform approximate identity $\Pc=(P_n)$.
\begin{itemize}
\item The set $S(\Xb)$ is a closed subalgebra of $\lb{\Xb}$, it contains the identity 
			operator $I$, and $SN(\Xb)$ is a proper closed ideal of $S(\Xb)$. 
\item The algebra $S(\Xb)$ is closed with respect to $s$-convergence, this 
			means that if $(A_n)\subset S(\Xb)$ s-converges to $A$ then 
			$A\in S(\Xb)$. 
\item $(A_n)\subset\lb{\Xb}$ has the $s$-limit $A$ iff it is bounded and
			$\|P_m(A_n-A)\|\to 0 \text{ for every }m.$
\item Let $FS(\Xb)$ denote the collection of all sequences 
			$(A_n)\subset \lb{\Xb}$ which possess a $s$-limit in $S(\Xb)$. Then
			the $s$-limit of every $(A_n)\in FS(\Xb)$ is uniquely determined.
\item Provided with pointwise operations and the supremum norm, $FS(\Xb)$ becomes 
			a Banach algebra with identity $\Ib:=(I)$. 
			The mapping $FS(\Xb) \to S(\Xb)$ which sends $(A_n)$ to its $s$-limit 
			$A$ is a unital algebra homomorphism and
			\begin{equation}\label{EPLimAbs}
			\|A\| \leq B_\Pc\liminf_{n\to\infty} \|A_n\|\quad\text{where}\quad
			B_\Pc:=\limsup_{n\to\infty}\|P_n\|.
			\end{equation}
\end{itemize}
\end{prop}
\begin{proof}
The first assertion is \cite[Lemmata 3.32, 3.3 and Corollary 3.5]{LiChW}.
For the second assertion let $K\in SN(\Xb)$. Then $KA_n\in SN(\Xb)$ and 
$\|KA_n-KA\| \to 0$. Since $SN(\Xb)$ is closed, this implies that $KA$ belongs 
to $SN(\Xb)$, thus $A\in S(\Xb)$. 
The \textit{if} part of the third assertion easily follows from the estimate
$\|K(A_n-A)\|\leq\|KQ_m\|\|(A_n-A)\|+\|K\|\|P_m(A_n-A)\|.$
The boundedness in the \textit{only if} part can be proved as 
\cite[Lemma 4.2]{LiChW} and the rest is trivial since $\Pc\subset SN(\Xb)$.
The fourth assertion is quite obvious: If $\|P_m(A_n-A)\|$ and $\|P_m(A_n-B)\|$ tend 
to zero as $n$ goes to infinity for every fixed $m$, then $P_m(A-B)=0$ for every $m$, 
hence $A-B=0$ since $\Pc$ is an approximate identity.

The proof of $FS(\Xb)$ being a normed algebra is again straightforward, and 
we only note that if $(A_n), (B_n)$ $s$-converge to $A, B\in S(\Xb)$, 
resp., then they are bounded and
\begin{align*}
\|K(A_nB_n-AB)\|&\leq \|K(A_n-A)B_n\| + \|KA(B_n-B)\| \to 0
\end{align*}
for every $K\in SN(\Xb)$, as $n\to\infty$, since $A\in S(\Xb)$ implies $KA\in SN(\Xb)$.
For the estimate~\eqref{EPLimAbs} we can replace $(A_n)$ by one of its subsequences 
which realize the $\liminf$, cancel arbitrarily but finitely many entries at the 
beginning of this subsequence and get from the third assertion that 
$\|P_mA\|\leq B_\Pc\liminf \|A_n\|+\epsilon$ for every $\epsilon>0$ and every $m$. 
Then we conclude \eqref{EPLimAbs}
since $\Pc$ is an approximate identity.
Finally, let $((C_n^m))_m$ be a Cauchy sequence of sequences $(C_n^m)\in FS(\Xb)$, 
where $C^m$ shall denote the $s$-limit of $(C_n^m)$, respectively. For every 
$n$, $(C_n^m)_m$ converges in $\lb{\Xb}$ to an element $C_n$, and the sequence $(C_n)$ is 
uniformly bounded. Furthermore, \eqref{EPLimAbs} yields that $(C^m)_m$ is a Cauchy 
sequence with a limit $C\in S(\Xb)$. Now one easily checks that $(C_n)$ $s$-converges 
to $C$, thus $FS(\Xb)$ is complete.
\end{proof}

Notice that in cases of $(P_n)$ being a sequence of compact operators and such that
their adjoints $P_n^*$ converge strongly to the identity we do not get anything 
new for Fredholm theory, since then $SN(\Xb)=\lc{\Xb}$ and $S(\Xb)=\lb{\Xb}$.
This particularly involves the cases $\Xb=l^p(\Zb^N,X)$ with $p\in\{0\}\cup(1,\infty)$ 
and $\dim X<\infty$.

\subparagraph{On inverse closedness of $S(\Xb)$ and generalized Fredholmness}
The Open Problem No. 3 in \cite{LiChW} asks whether $S(\Xb)$ is inverse closed.
In the above mentioned cases where $S(\Xb)=\lb{\Xb}$ this is obviously true.
Another partial answer in \cite[Section 3.2]{LiChW} extends that to the cases of
compact $(P_n)$ on Banach spaces $\Xb$ which are complete w.r.t. a certain topology.
By this means, one gets an affirmative answer for all cases $\Xb=l^p(\Zb^N,X)$ 
with $p\in\{0\}\cup[1,\infty]$ and $\dim X<\infty$. The most popular application
with $\dim X=\infty$ is $\Xb=l^p(\Zb,L^p(0,1))\cong L^p(\Rb)$. Unfortunately, 
already in this natural situation the picture changes, as the following examples 
demonstrate.

\begin{ex}\label{Ex1}
Consider $\Xb=l^\infty(\Zb,L^\infty(0,1))$. For $k\in\Nb$ define
\[B_k:L^\infty\left(\frac{1}{k+1},\frac{1}{k}\right)\to L^\infty(0,1),\quad B_kf(x)=f\left(\frac{1}{k+1}+x\left(\frac{1}{k}-\frac{1}{k+1}\right)\right)\]
and the respective extensions $C_k:L^\infty(0,1)\to L^\infty(0,1)$,
$C_k=B_k\chi_{\left(\frac{1}{k+1},\frac{1}{k}\right)}I$.
Figuratively speaking, these operators single out a certain part of a given function
and stretch it. Now, let the operator $A$ on $\Xb$ be given by
\[(Ax)_n=\begin{cases}
x_n & : n< 0\\
C_{\frac{n}{2}+1}x_0 &: n\geq 0 \text{ even}\\
x_{\frac{n+1}{2}} &: n\geq 0 \text{ odd}.
\end{cases}\]
Its matrix representation is shown in Figure \ref{FEx}.
Clearly, this operator belongs to $S(\Xb)$ and it is invertible where its inverse
can be obtained by reflecting the matrix w.r.t. the main diagonal and replacing 
$C_n$ by the inverse of $B_n$. This inverse does not belong to $S(\Xb)$.
\end{ex}
\begin{figure}
$\begin{pmatrix}
\ddots & & & & & & & &\\
& I & & & & & & &\\
& & I & & & & & &\\
& & & C_1 & & & & &\\
& & & 0 & I & & & &\\
& & & C_2 & 0 & 0 & & &\\
& & & 0 & 0 & I & 0 & &\\
& & & C_3 & 0 & 0 & 0 & 0 &\\
& & & 0 & 0 &  0& I & 0 &\\
& & & C_4 & & & & &\\
& & & \vdots & & & & \vdots &\\
\end{pmatrix},
\quad \quad
\begin{pmatrix}
\ddots & & & \vdots & & & &\\
& J_2 & & J_1B_{-2} & & & &\\
& & J_2 & J_1B_{-1} & & & &\\
& & & B_0 & & & &\\
& & & J_1B_1 & J_2 & & &\\
& & & J_1B_2 & & J_2 & &\\
& & & J_1B_3 & & & J_2 &\\
& & & \vdots & & & & \ddots\\
\end{pmatrix}.
$
\caption{Matrix representations of the Examples \ref{Ex1} and \ref{Ex2}.}
\label{FEx}
\end{figure}
\begin{ex}\label{Ex2}
Let $\theta_1:L^p(0,1)\to L^p(0,1/2)$ and $\theta_2:L^p(0,1)\to L^p(1/2,1)$ be 
isometric isomorphisms, define 
$J_1:=\diag(\ldots, \theta_1,\theta_1, I, \theta_1, \theta_1, \ldots)$ and
$J_2:=\diag(\ldots, \theta_2,\theta_2, 0, \theta_2, \theta_2, \ldots)$,
and recall the operators $\psi$ and $B$ from Example \ref{EExpand}. Then
$A:=J_1B+J_2$ is an invertible isometry on $\Xb:=l^p(\Zb,L^p(0,1))$ which belongs
to $S(\Xb)$, but $A^{-1}\notin S(\Xb)$. With 
the definition $B_k:=\chi_{(0,1)}V_{-k}\psi:L^p(0,1)\to L^p(0,1)$ its matrix 
representation is shown in Figure \ref{FEx}.
\end{ex}

Notice that one cannot expect a comparable Fredholm theory as well. Clearly, every 
invertible operator $A\in S(\Xb)$ is regularizable w.r.t. $SN(\Xb)$ in the spirit 
of Definition \ref{DInvAtInf} (invertible at infinity), but it may not belong to an 
invertible coset in $S(\Xb)/SN(\Xb)$ (as in Definition \ref{DPFred}) in general 
since the inverse may be outside $S(\Xb)$.
Actually, the ``invertibility at infinity'' would not even be compatible with 
multiplication: The operator $A$ from Example \ref{Ex1} is invertible, hence 
``invertible at infinity'', and so is $Q_1$, but $AQ_1$ is not: Assume that there is 
a $SN(\Xb)$-regularizer $B$ for $AQ_1$ and let $R$ denote the projection 
\[(x_n)\mapsto (y_n),\quad y_n:=
\begin{cases} x_n &: n\geq 0\text{ even}\\ 0 &: \text{otherwise}.\end{cases}\]
Then $R\in S(\Xb)$ and, due to the equality $R=R(I-AQ_1B)$, it would even belong to 
$SN(\Xb)$, a contradiction. Moreover this yields that $A$ has no regularizer w.r.t. 
$SN(\Xb)$ in $S(\Xb)$ at all: To see this, assume that $C\in S(\Xb)$ is a regularizer 
for $A$, then 
$I-AQ_1C=I-AC+AP_1C\in SN(\Xb)$ and $I-CAQ_1=I-CA+CAP_1\in SN(\Xb)$, that means $C$ is a 
regularizer for $AQ_1$, again a contradiction. Of course, a similar observation can
easily be made for the operators on $l^p$-spaces from Example \ref{Ex2} as well.

Thus, we see that it is hardly possible to achieve a comparable Fredholm theory in
the $S(\Xb)$-setting without further restrictions and assumptions.

\subparagraph{On limit operators}
Given an operator $A$ one has to expect a larger operator spectrum after passing
to a weaker definition of convergence. However,  
as a start we point out that $s$-convergence cannot provide any additional benefit 
as long as one is only interested in the limit operators of band-dominated operators.
For this, we denote by $\sigma^\Pc_{\op}(A)$, $\sigma^s_{\op}(A)$ the operator spectra 
of $A$ w.r.t. $\Pc$-strong convergence or $s$-convergence, respectively.

\begin{prop}
For every $A\in\Alp$ we have $\sigma^\Pc_{\op}(A)=\sigma^s_{\op}(A)$.
\end{prop}
\begin{proof}
The inclusion ``$\subset$'' is obvious. 
Assume that $A_g\in\sigma^s_{\op}(A)\setminus\sigma^\Pc_{\op}(A)$. Then there is a 
subsequence $h$ of $g$ such that $\|(V_{-h_n}AV_{h_n}-A_g)P_m\|\to c>0$. Since $A$ 
and $A_g$ are band-dominated there exist $k$, $m$ and $n_0$ such that 
$\|Q_k(V_{-h_n}AV_{h_n}-A_g)P_m\|<c/2$ for all $n\geq n_0$. To see this just approximate
$A$ an $A_g$ by band operators. This yields $\|P_k(V_{-h_n}AV_{h_n}-A_g)\|\not\to 0$,
a contradiction.
\end{proof}

So, let us now look at operators outside $\Alp$. When the $\Pc$-strong convergence is 
replaced by $s$-convergence, the operator spectrum of invertible operators becomes larger
and, unfortunately, may contain non-invertible limit operators, as the next example shows. 
Thus, Theorem \ref{TLimOps}, the backbone of the limit operator method, does not remain valid 
for $s$-convergence.

\begin{ex}
Consider $\Xb=l^p(\Zb,\Cb)$ with the canonical projections $(P_n)$, let $I_m$ denote 
the $m\times m$ identity matrix and $C_m$ the $m\times m$ circulant matrix
\[C_m:=\begin{pmatrix}
0 & 1 & &\\
& 0 & \ddots& \\
& & \ddots & 1 \\
1 & & & 0 \\
\end{pmatrix}.
\quad \text{ Also set }
B:=
\begin{pmatrix}
\ddots & & & & & &\\
& 1 & & & & & &  \\
& & 1 & & & & &  \\
& & & 0 & 1 & &  \\
& & & & 0 & 1 &  \\
& & & & & 0 & \ddots& \\
& & & & & & \ddots\\
\end{pmatrix}.\]
Now define $A$ by the infinite block diagonal matrix 
$\diag\{Q,I_1,C_1,I_2,C_2,I_3,C_3,\ldots\}$.
Then $A$ belongs to $\pb{\Xb}$, is invertible, but has the non-invertible limit operator
$B\in\sigma^s_{\op}(A)$.
\end{ex}

\subparagraph{Quasi-banded operators}
Although the above examples do not condemn the hope for a reasonable Fredholm or limit 
operator theory w.r.t. $s$-convergence in total, they at least show that important 
results are not available in this more general $S(\Xb)$-setting.
However, there is still a wide field of possible generalizations between the pleasant 
and well understood class $\Alp$ of band-dominated operators and the whole algebra 
$\pb{\Xb}$, equipped with $\Pc$-strong convergence. One should draw attention e.g. to 
the following class of operators:
\begin{defn}
An operator $A\in\lb{\Xb}$ is said to be quasi-banded if
\begin{equation}\label{DWBDO}
\lim_{m\to\infty}\sup_{n>0}\|Q_{n+m}AP_n\| =
\lim_{m\to\infty}\sup_{n>0}\|P_n AQ_{n+m}\| = 0.
\end{equation}
\end{defn}
As in \cite{MaSaSe} one can verify that the set $\Qc_p$ of all quasi-banded operators 
is a Banach algebra, $\Ac_p\subset\Qc_p\subset\pb{\Xb}$ and the inclusions are proper 
in general. In particular, the flip operator is quasi-banded, and so are also e.g. 
Laurent operators with quasi-continuous or slowly oscillating symbols.
Moreover, $\Qc_p$ is inverse closed and closed under passing to $\Pc$-regularizers. 

In the case $N=1$, a shift invariant operator $A$ belongs to $\Qc_p$ if and only if 
\begin{equation}\label{EWBDOChar}
PAQ \text{ and }QAP\text{ are $\Pc$-compact}.
\end{equation}
Moreover, Theorem \ref{TIndex} and, in particular, the index formula \eqref{EIndExt} 
also hold for operators in $\Qc_p$.

The most striking argument for the consideration of these operators is the observation 
that the well known an intensively studied results on the applicability of the finite
section method in terms of limit operators extend from $\Ac_p$ to $\Qc_p$. This is 
subject of \cite{MaSaSe} for the case $l^p(\Zb,L^p(0,1))$, but the arguments there 
also work for more general situations $l^p(\Zb,X)$, and suggest analogues for 
the spaces $l^p(\Zb^N,X)$.

\section{Conclusions}

Within this paper we have presented an overview of the recent state of the art in
Fredholm theory for band-dominated and related operators which is based on the
beautiful and most complete $\Pc$-approach. It turned out that replacing the 
classical functional analytic approach and the notions of compactness, Fredholmness 
and strong convergence by the respective $\Pc$-triple provides a Banach algebraic
framework which perfectly fits to the problems one is interested in. It permits to 
treat the whole scale
of spaces $l^p(\Zb^N,X)$ (and even more) in a completely homogeneous way. In fact,
there is no need for reflexivity, for a predual setting \eqref{CPrae}, for the 
hyperplane property of $X$ \eqref{CHyper}, or for any other restrictions on $X$, e.g. on
its dimension. We have particularly answered the Open Problems No. 4 and 5 of \cite{LiChW}.
Moreover, the conjectures No. 7 (Sufficiency of Favards condition for $N>1$) and 3 
(inverse closedness of $S(\Xb)$ turned out to be wrong. We have also seen that the 
important partial answers to the big question (Problem No. 8) actually hold for all 
rich operators in the framework $\pb{l^p}$, $p\in\{0,1,\infty\}$, and are not specific 
for band-dominated operators. 

Concerning the Open Problems No. 1 and 2 we could make the following observations: 
The passage from the classical approach to $\pb{\Xb}$ turned out to be a great step.
The algebraic and topological structures and relations remained essentially the same, 
this $\Pc$-framework is consistent and self-contained, and one could achieve a much
higher flexibility and generality. Also the proofs became more transparent. The 
possible next step towards $S(\Xb)$ and $s$-convergence is less promising since
several important basic results which are necessary for a Fredholm theory and the
limit operator method cannot be extended in general, and will not provide a 
comparable toolbox.


\end{document}